\documentclass[a4paper,twoside]{article}
\usepackage{a4}
\usepackage{amssymb}
\usepackage{amsmath}
\usepackage{upref}
\usepackage[active]{srcltx}
\usepackage[dvips,colorlinks,citecolor=blue,linkcolor=blue]{hyperref}
\usepackage[dvipsnames]{color}
\allowdisplaybreaks[2] % To make it possible for displaybreaks
\newcount\minutes \newcount\hours
\hours=\time
\divide\hours 60
\minutes=\hours
\multiply\minutes -60
\advance\minutes \time
\newcommand{\klockan}{\the\hours:{\ifnum\minutes<10 0\fi}\the\minutes}
\newcommand{\tid}{\today\ \klockan}
\newcommand{\prtid}{\smash{\raise 10mm \hbox{\LaTeX ed \tid}}}
\renewcommand{\prtid}{}
\makeatletter
\def\sectionmark#1{} %\markboth{{\sectnr #1}}{{\sectnr #1}}} %Journal
\def\subsectionmark#1{}
\newcommand{\sectnr}{\ifnum \c@secnumdepth >\z@
                 \thesection.\hskip 1em\relax \fi}
\def\@evenhead{\footnotesize\rm\thepage\hfil\leftmark\hfil\llap{\prtid}}
\def\@oddhead{\footnotesize\rm\rlap{\prtid}\hfil\rightmark\hfil\thepage}
\def\tableofcontents{\section*{Contents} %\@mkboth{Contents}{Contents}} %Journal
 \@starttoc{toc}}
\makeatother
\makeatletter
\def\@biblabel#1{#1.}
\makeatother
\makeatletter
\let\Thebibliography=\thebibliography
\renewcommand{\thebibliography}[1]{\def\@mkboth##1##2{}\Thebibliography{#1}
\addcontentsline{toc}{section}{References}
\frenchspacing % Maybe not needed
\setlength{\@topsep}{0pt}% Delete if extra space before list
\setlength{\itemsep}{0pt}%
\setlength{\parskip}{0pt plus 2pt}%
}
\makeatother
\makeatletter
\def\ldots{\mathinner.\nonscript\!.%
 \ifx\next,.\else\ifx\next;.\else\ifx\next..\else
 \nonscript\!\mathinner.\fi\fi\fi}
\makeatother
\makeatletter
\let\Enumerate=\enumerate
\renewcommand{\enumerate}{\Enumerate%
\setlength{\@topsep}{0pt}% Delete if extra space before list
\setlength{\itemsep}{0pt}%
\setlength{\parskip}{0pt plus 1pt}%
\renewcommand{\theenumi}{\textup{(\alph{enumi})}}%
\renewcommand{\labelenumi}{\theenumi}%
}
\let\endEnumerate=\endenumerate
\renewcommand{\endenumerate}{\endEnumerate\unskip}
\makeatother
\makeatletter
\def\@seccntformat#1{\csname the#1\endcsname.\quad}
\makeatother
\makeatletter
\long\def\@makecaption#1#2{%
  \vskip\abovecaptionskip
  \sbox\@tempboxa{ #1. #2}%
  \ifdim \wd\@tempboxa >\hsize
    #1. #2\par
  \else
    \global \@minipagefalse
    \hb@xt@\hsize{\hfil\box\@tempboxa\hfil}%
  \fi
  \vskip\belowcaptionskip}
\makeatother
\newcommand{\authortitle}[2]{\author{#1}\title{#2}\markboth{#1}{#2}}
\newcommand{\art}[6]{{\sc #1, \rm #2, \it #3\/ \bf #4 \rm (#5), \mbox{#6}.}}
\newcommand{\auth}[2]{{#1,  #2.}}
\def\idxauth{\auth}
\newcommand{\artprep}[3]{{\sc #1, \rm #2, #3.}}
\newcommand{\book}[3]{{\sc #1, \it #2, \rm #3.}}
\newcommand{\AND}{{\rm and }}
\newcommand{\artnopt}[6]{{\sc #1, \rm #2, \it #3\/ \bf #4 \rm (#5), \mbox{#6}}}
\RequirePackage{amsthm}
\newtheoremstyle{descriptive}%
  {\topsep}   %{\medskipamount}          % Space above
  {\topsep}   %  {\medskipamount}          % Space below
  {\rmfamily} % Body font
  {}          % Indent
  {\bfseries} % Head font
  {.}         % Punctuation after thm head
  { }         % Space after thm head
  {}          % Thm head spec(?)
\newtheoremstyle{propositional}%
  {\topsep}   %  {\medskipamount}          % Space above
  {\topsep}   %  {\medskipamount}          % Space below
  {\itshape}  % Body font
  {}          % Indent
  {\bfseries} % Head font
  {.}         % Punctuation after thm head
  { }         % Space after thm head
  {}          % Thm head spec(?)
\theoremstyle{propositional}
\newtheorem{thm}{Theorem}[section]
\newtheorem{prop}[thm]{Proposition}
\newtheorem{lem}[thm]{Lemma}
\theoremstyle{descriptive}
\newtheorem{deff}[thm]{Definition}
\newtheorem{example}[thm]{Example}
\newtheorem{remark}[thm]{Remark}
\makeatletter
\renewenvironment{proof}[1][\proofname]{\par
  \pushQED{\qed}%
  \normalfont
  \trivlist
  \item[\hskip\labelsep
        \itshape
    #1\@addpunct{.}]\ignorespaces
}{%
  \popQED\endtrivlist\@endpefalse
}
\makeatother
\newcommand{\setm}{\setminus}
\newcommand{\Cp}{{C_p}}
\DeclareMathOperator{\diam}{diam}
\DeclareMathOperator{\Div}{div}
\DeclareMathOperator{\dist}{dist}
\newcommand{\bdy}{\partial}
{\catcode`p =12 \catcode`t =12 \gdef\eeaa#1pt{#1}}      % Get slantfactor
\def\accentadjtext#1{\setbox0\hbox{$#1$}\kern   % Convert it with height
                \expandafter\eeaa\the\fontdimen1\textfont1 \ht0 }
\def\accentadjscript#1{\setbox0\hbox{$#1$}\kern % Convert it with height
                \expandafter\eeaa\the\fontdimen1\scriptfont1 \ht0 }
\def\accentadjscriptscript#1{\setbox0\hbox{$#1$}\kern   % Convert it with height
                \expandafter\eeaa\the\fontdimen1\scriptscriptfont1 \ht0 }
\def\accentadjtextback#1{\setbox0\hbox{$#1$}\kern       % Convert it with height
                -\expandafter\eeaa\the\fontdimen1\textfont1 \ht0 }
\def\accentadjscriptback#1{\setbox0\hbox{$#1$}\kern     % Convert it with height
                -\expandafter\eeaa\the\fontdimen1\scriptfont1 \ht0 }
\def\accentadjscriptscriptback#1{\setbox0\hbox{$#1$}\kern % Convert it with height
                -\expandafter\eeaa\the\fontdimen1\scriptscriptfont1 \ht0 }
\def\itoverline#1{{\mathsurround0pt\mathchoice
        {\rlap{$\accentadjtext{\displaystyle #1}
                \accentadjtext{\vrule height1.593pt}
                \overline{\phantom{\displaystyle #1}
                \accentadjtextback{\displaystyle #1}}$}{#1}}
        {\rlap{$\accentadjtext{\textstyle #1}
                \accentadjtext{\vrule height1.593pt}
                \overline{\phantom{\textstyle #1}
                \accentadjtextback{\textstyle #1}}$}{#1}}
        {\rlap{$\accentadjscript{\scriptstyle #1}
                \accentadjscript{\vrule height1.593pt}
                \overline{\phantom{\scriptstyle #1}
                \accentadjscriptback{\scriptstyle #1}}$}{#1}}
        {\rlap{$\accentadjscriptscript{\scriptscriptstyle #1}
                \accentadjscriptscript{\vrule height1.593pt}
                \overline{\phantom{\scriptscriptstyle #1}
                \accentadjscriptscriptback{\scriptscriptstyle #1}}$}{#1}}}}
\def\itunderline#1{{\mathsurround0pt\mathchoice
        {\rlap{$\underline{\phantom{\displaystyle #1}
                \accentadjtextback{\displaystyle #1}}$}{#1}}
        {\rlap{$\underline{\phantom{\textstyle #1}
                \accentadjtextback{\textstyle #1}}$}{#1}}
        {\rlap{$\underline{\phantom{\scriptstyle #1}
                \accentadjscriptback{\scriptstyle #1}}$}{#1}}
        {\rlap{$\underline{\phantom{\scriptscriptstyle #1}
                \accentadjscriptscriptback{\scriptscriptstyle #1}}$}{#1}}}}
\newcommand{\alp}{\alpha}
\newcommand{\ga}{\gamma}
\newcommand{\eps}{\varepsilon}
\newcommand{\Om}{\Omega}
\newcommand{\clcombm}{{\overline{\comb}\mspace{1mu}}^M}
\newcommand{\combG}{{\comb^{G}}}
\renewcommand{\phi}{\varphi}
\newcommand{\phit}{\widetilde{\varphi}}
\newcommand{\p}{{$p\mspace{1mu}$}}
\newcommand{\R}{\mathbf{R}}
\newcommand{\eR}{{\overline{\R}}}
\newcommand{\ud}{\mathbf{D}}
\newcommand{\limplus}{{\mathchoice{\raise.17ex\hbox{$\scriptstyle +$}}
                {\raise.17ex\hbox{$\scriptstyle +$}}
                {\raise.1ex\hbox{$\scriptscriptstyle +$}}
                {\scriptscriptstyle +}}}
\newcommand{\uHp}{\itoverline{P}}   % upper solution
\newcommand{\lHp}{\itunderline{P}}  % lower solution
\newcommand{\Hp}{P}                 % upper=lower solution
\newcommand{\Hpind}[1]{P_{#1}}      % with domain
\newcommand{\uHpind}[1]{\itoverline{P}_{#1}}      % with domain
\newcommand{\lHpind}[1]{\itunderline{P}_{#1}}      % with domain
\newcommand{\EHp}{P^{\rm Ext}}      % with domain
\newcommand{\EuHp}{\itoverline{P}^{\rm Ext}}      % with domain
\newcommand{\ElHp}{\itunderline{P}^{\rm Ext}}      % with domain
\newcommand{\SHp}{S}      % with domain
\newcommand{\SuHp}{\itoverline{S}}      % with domain
\newcommand{\SlHp}{\itunderline{S}}      % with domain
\newcommand{\wt}{\widetilde{w}}
\newcommand{\psit}{\tilde{\psi}}
\newcommand{\vt}{\tilde{v}}
\newcommand{\eith}{e^{i\theta}}
\newcommand{\reith}{r\eith}
\newcommand{\dM}{d_M}
\newcommand{\Ext}{{\rm Ext}}
\newcommand{\bdyExt}{\bdy_{\rm Ext}}
\newcommand{\Extbdy}{\bdyExt}
\newcommand{\combExt}{\comb^{\rm Ext}}
\newcommand{\bdyM}{\partial_{M}}
\newcommand{\bdyP}{\partial_{P}}
\newcommand{\clOmP}{{\overline{\Om}\mspace{1mu}}^P}
\newcommand{\UU}{\mathcal{U}}%
\newcommand{\UUt}{\widetilde{\mathcal{U}}}%
\newcommand{\Cbdd}{C_{\rm bdd}}
\newcommand{\Cunif}{C_{\rm unif}}
\newcommand{\ktilde}{\tilde{k}}
\newcommand{\kt}{\tilde{k}}
\newcommand{\bCp}{{\protect\itoverline{C}_p}}
\newcommand{\comb}{\Psi}
\newcommand{\cprime}{$'$}
\newcommand{\Et}{\widetilde{E}}
\newcommand{\Omt}{\widetilde{\Om}}
\numberwithin{equation}{section}
\newenvironment{ack}{\medskip{\it Acknowledgement.}}{}

\begin{document}

\authortitle{Anders Bj\"orn}
        {The Dirichlet problem for \p-harmonic functions on the 
topologist's comb}
\author{
Anders Bj\"orn \\
\it\small Department of Mathematics, Link\"opings universitet, \\
\it\small SE-581 83 Link\"oping, Sweden\/{\rm ;}
\it \small anders.bjorn@liu.se
}

\date{}
\maketitle

\noindent{\small
{\bf Abstract}.
In this paper we study the Perron method for solving
the \p-harmonic Dirichlet problem on the topologist's comb.
For functions which are bounded and continuous 
at the accessible points, 
we obtain invariance of the Perron solutions under arbitrary
perturbations on the set of inaccessible points.
We also obtain some results allowing for jumps and perturbations at a
countable set of points.
}

\bigskip


\medskip


\section{Introduction}

In the Dirichlet problem one looks for a \p-harmonic function $u$
on some bounded domain $\Om \subset \R^n$ which takes prescribed boundary values $f$.
A \emph{\p-harmonic function} $u$ is a continuous weak solution  of the
equation
\[
      \Div(|\nabla u|^{p-2}\nabla u)=0.
\]
(And thus for $p=2$ we obtain the usual harmonic functions.)
Here $1<p<\infty$ is fixed. 
The nonlinear potential theory associated with
\p-harmonic functions has been studied for half a century,
first on $\R^n$ and then in various other situations (manifolds,
Heisenberg groups, graphs etc.),
and more recently on metric spaces
giving 
a unified treatment covering most of the earlier cases,
see the monographs
Heinonen--Kilpel\"ainen--Martio~\cite{HeKiMa} (for weighted $\R^n$)
and Bj\"orn--Bj\"orn~\cite{BBbook} (for metric spaces)
and the references therein.

If $f$ is not continuous, then there usually is no \p-harmonic
function $u$ which takes the boundary values as limits
(i.e.\ such that $\lim_{y \to x} u(y)=f(x)$ for all $x \in \partial \Om$),
and even for continuous $f$ and with $p=2$ this is not always possible.
One therefore needs some other precise definition of what
is a \emph{solution} to the Dirichlet problem.
For \p-harmonic functions there are at least four different
definitions in the literature, of which the \emph{Perron method} is the most general,
see the definitions in 
Bj\"orn--Bj\"orn--Shanmugalingam~\cite{BBS}, \cite{BBS2} and
Bj\"orn--Bj\"orn~\cite{BB2} as well as Theorem~4.2 in~\cite{BB2},
or the discussion in the introduction to Chapter~10
in Bj\"orn--Bj\"orn~\cite{BBbook}.

For any boundary function $f: \partial \Om \to \eR:=[-\infty,\infty]$,
the Perron method produces an upper and a lower Perron solution.
When these coincide they give a reasonable solution to the Dirichlet problem,
called the \emph{Perron solution} $\Hp f$, 
and $f$ is said to be \emph{resolutive},
see Section~\ref{sect-Perron} for the precise definition.

In this paper we want to study the Dirichlet problem,
or more precisely  Perron solutions, 
for \p-harmonic functions 
on the toplogist's comb  
\[
    \comb= ((-1,1)\times(0,2)) \setm \bigcup_{j=0}^\infty \itoverline{I}_j
\]
in the plane,
where     $I_j= (0,1) \times \{2^{-j}\}$, $j=0,1, \ldots$,
see Figure~\ref{fig1}.
Let $I= (0,1]  \times \{0\}$ be the set of inaccessible boundary points of $\comb$.

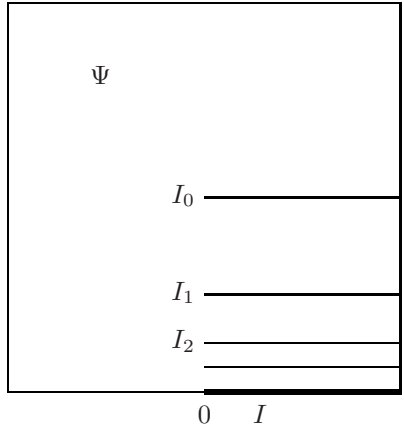
\begin{figure}[t]
\begin{center}
\setlength{\unitlength}{2000sp}%
\begingroup\makeatletter\ifx\SetFigFont\undefined%
\gdef\SetFigFont#1#2#3#4#5{%
  \reset@font\fontsize{#1}{#2pt}%
  \fontfamily{#3}\fontseries{#4}\fontshape{#5}%
  \selectfont}%
\fi\endgroup%
\begin{picture}(4834,5176)(2379,-5525)
\thinlines
{\color[rgb]{0,0,0}\put(4801,-2761){\line( 1, 0){2400}}
}%
{\color[rgb]{0,0,0}\put(2401,-5161){\framebox(4800,4800){}}
}%
{\color[rgb]{0,0,0}\put(4801,-3961){\line( 1, 0){2400}}
}%
{\color[rgb]{0,0,0}\put(4801,-4561){\line( 1, 0){2400}}
}%
{\color[rgb]{0,0,0}\put(4801,-4861){\line( 1, 0){2400}}
}%
\thicklines%
\linethickness{2pt}%
{\color[rgb]{0,0,0}\put(4801,-5161){\line( 1, 0){2415}}
}%
\put(3401,-1361){\makebox(0,0)[lb]{\smash{{{$\comb$}%
}}}}
\put(4401,-2821){\makebox(0,0)[lb]{\smash{{{$I_0$}%
}}}}
\put(4401,-4021){\makebox(0,0)[lb]{\smash{{{$I_1$}%
}}}}
\put(4401,-4621){\makebox(0,0)[lb]{\smash{{{$I_2$}%
}}}}
\put(5401,-5541){\makebox(0,0)[lb]{\smash{{{$I$}%
}}}}
\put(4721,-5541){\makebox(0,0)[lb]{\smash{{{$0$}%
}}}}
\end{picture}%
\end{center}
\caption{\label{fig1}%
The topologist's comb $\comb$.}
\end{figure}

We obtain the following result,
which is a special case of Theorem~\ref{thm-main}.

\begin{thm} \label{thm-intro}
Let $f: \bdy \comb \to \eR$ be such that
$f|_{\bdy \comb \setm I} \in \Cbdd(\bdy \comb \setm I)$.
Then $f$ is resolutive, and 
the Perron solution $\Hp f$  is independent of
the values of $f$ on $I$, i.e.\
if $h=f$ on $\bdy \comb \setm I$,
then $\Hp h= \Hp f$.
\end{thm}

(In the linear case, $p=2$, this is well known and can be obtained more easily.)

The Perron method was introduced independently by 
Perron~\cite{perron} and Remak~\cite{remak} in the 1920s
for harmonic functions.
The linear theory was developed further by Wiener and Brelot,
and the method is therefore 
often called the PWB method in the linear case.
In the nonlinear case the theory was
developed by
Granlund--Lindqvist--Martio~\cite{GrLiMa86},
Kilpel\"ainen~\cite{Kilp89}
and 
Heinonen--Kilpel\"ainen--Martio~\cite{HeKiMa}
for unweighted and weighted $\R^n$.
In particular the resolutivity was obtained 
for continuous $f: \bdy \Om \to \R$ for arbitrary bounded domains
$\Om \subset \R^n$ (in the unweighted case in \cite{Kilp89}
and in the weighted case in \cite{HeKiMa}).

The first invariance result of the kind above 
(in the nonlinear case)
was obtained in
Bj\"orn--Bj\"orn--Shanmugalingam~\cite{BBS2}
where it was shown that if
$f \in C(\bdy \Om)$ and $h=f$ 
outside a set of \p-capacity zero, then
$h$ is resolutive and 
$\Hp h = \Hp f$.
This was obtained for bounded domains  $\Om$
in metric measure spaces (under the usual assumptions
that the metric space is complete and the measure is doubling and supports
a \p-Poincar\'e inequality).
In Bj\"orn--Bj\"orn~\cite{BBbook} this result
was improved slightly by allowing for a (sometimes) smaller capacity.
More recently, 
in Bj\"orn--Bj\"orn--Shanmugalingam~\cite{BBSdir},
it was further improved using again a (sometimes) smaller capacity
$\bCp(\,\cdot\,,\Om)$ introduced therein,
which sees the boundary from inside $\Om$ 
(see \cite{BBSdir} for the precise definition).
In particular, it was shown in Example~10.2 % Check no
in \cite{BBSdir}
that 
\begin{equation} \label{eq-bCpI}
     \bCp(I,\comb)=0,
\end{equation}
so that
Theorem~\ref{thm-intro} was obtained therein for functions $f$ for which 
there exists $k \in C(\bdy \comb)$ such that $k=f$ on $\bdy \comb \setm I$,
i.e.\ $f$ such that $f|_{\bdy \comb \setm I} \in \Cunif(\bdy \comb \setm I)$.

The significance of Theorem~\ref{thm-intro} is that we do not assume any continuity
at points in $I$, or more precisely consider functions
in $\Cbdd(\bdy \comb \setm I)$.
That Theorem~\ref{thm-intro} is not true
for unbounded functions in 
$C(\bdy \comb \setm I)$ is shown in Example~\ref{ex-1},
as such functions need not be resolutive.

In Theorem~\ref{thm-main} we obtain a generalization of Theorem~\ref{thm-intro}
which is connected with the prime end boundary of $\comb$.
Here it is not the classical prime end boundary of Carath\'eodory~\cite{car}
which is used. Instead it is the prime end definition introduced
in 
Adamowicz--Bj\"orn--Bj\"orn--Shan\-mu\-ga\-lin\-gam~\cite{ABBSprime}
which is the natural choice in this paper.
The noncompactness of the prime end closure of the comb
leads to some new phenomena, see Section~\ref{sect-bdy-reg}.
In domains which are so-called finitely connected at the boundary,
the prime end closure is compact and the theory
of Perron solutions with respect to the prime end boundary for
such domains was developed in 
Bj\"orn--Bj\"orn--Shanmugalingam~\cite{BBSdir}.
Estep--Shanmugalingam~\cite{ES}
are studying similar problems when the prime end closure is
noncompact.

Let us compare our result with the unit disc $\ud$ in the plane
and let $x_0=(1,0)$. 
Let also $f: \bdy \ud \to \eR$ be a function such that 
$f|_{\bdy \ud \setm \{x_0\}}$ is bounded and continuous.
If $f$ is semicontinuous then $f$ is resolutive (for this we need
to use that $\ud$ is a regular domain), 
see 
Proposition~9.31 in Heinonen--Kilpel\"ainen--Martio~\cite{HeKiMa}
and Proposition~7.3 in Bj\"orn--Bj\"orn--Shanmugalingam~\cite{BBS2}
(or Proposition~10.32 in \cite{BBbook}),
but if $f(x_0)$ is such 
that $f$ is not semicontinuous, then it is not  known if $f$ is resolutive.
Moreover, all choices of $f(x_0)$ which make $f$ upper semicontinuous
yield the same Perron solution, by 
Proposition~7.3 in \cite{BBS2}
(or Proposition~10.32 in \cite{BBbook}).
Similarly all choices of $f(x_0)$ which make $f$ lower  semicontinuous
yield the same Perron solution, but we do not know
if this Perron solution is the same as the one for upper semicontinuous
choices of $f(x_0)$.
If $f$ has a jump discontinuity at $x_0$, then we do know
that $f$ is resolutive for all choices of $f(x_0)$ and that the Perron solutions
all agree (i.e.\ are independent of $f(x_0)$), by 
Theorems~6.3 and 7.3 in Bj\"orn~\cite{ABjump}.
(For $p \le 2$ it is not too difficult to deduce this using
the earlier results in 
Bj\"orn--Bj\"orn--Shanmugalingam~\cite{BBS3}.)
Thus we have less general invariance results for perturbations on a single
point on the boundary of $\ud$ than those we obtain in this paper for perturbations
on $I$ on the boundary of the comb $\comb$. 
(Above the regularity of $\ud$ was important, but there are
some  results in this direction in  \cite{ABjump}  
which hold also for semiregular sets.)

The outline of the paper is as follows:
In Section~\ref{sect-comb} the comb and its various boundaries are introduced,
while in Section~\ref{sect-Perron} the Perron solutions considered in this paper 
are defined.
In Section~\ref{sect-bdy-reg} we obtain some boundary regularity results which
will be essential for us.
The main result (Theorem~\ref{thm-main}) is obtained in 
Section~\ref{sect-main}.
Finally in Section~\ref{sect-jumps} we combine the ideas
in this paper with some ideas in Bj\"orn~\cite{ABjump} to
obtain a generalization of our main result.

\begin{ack}
The author was supported by the Swedish Research Council.
Part of this research was done while the author
was a Fulbright scholar (supported by the Swedish
Fulbright Commission)  
visiting 
the University of Cincinnati in 2010.
He would like to thank Tomasz Adamowicz, Jana Bj\"orn and 
Nageswari Shanmugalingam for 
helpful discussions related to this research.
\end{ack}

\section{The toplogist's comb}
\label{sect-comb}

The aim of this paper is to study the Dirichlet problem on 
the comb $\comb$.
The boundary points of $\comb$ are of three different types that will be of interest
to us.

A boundary point $x_0 \in \bdy \comb$
is \emph{accessible} if there is a continuous mapping (a curve)
$\ga:[0,1] \to \overline{\comb}$ such that $\ga(1)=x_0$ and $\ga([0,1)) \subset \comb$.
The set $I$ consists of all the inaccessible boundary points.

The boundary points in $I_j$ each have two natural counterparts in the extended
boundary we shall define below, one by taking limits from below and one from above.
To be more precise,
set $\theta_j=2^{-j}$, $j=1,2,\ldots$, and define $F : \comb \to \R^3$ by letting
\[
   F(x_1,x_2)=\begin{cases}
     (x_1,x_2,0), & \text{if } (x_1,x_2) \in (-1,0] \times (0,2), \\
     (x_1 \cos \theta_j,x_2,x_1 \sin \theta_j), 
            & \text{if } (x_1,x_2) \in (0,1) \times (2^{-j},2^{1-j}) ,\ j=0,1,\ldots.
     \end{cases}
\]
Let $\combExt$ be $\comb$ equipped with the distance 
$\dist_{\rm Ext}(x,y) = |F(x)-F(y)|$, $x,y \in \comb$.
Let also $\bdyExt \comb=\overline{\combExt} \setm \comb$, 
where $\overline{\combExt}$ is the completion of $\combExt$.
Each point in $\bigcup_{j=0}^\infty I_j$ corresponds to two points in this extended
boundary, whereas all other points in $\bdy \comb$ have one counterpart in 
$\bdyExt \comb$.
Let also $\Phi:\bdyExt \comb \to \bdy \comb$ be the natural map.

The extended boundary is closely related to prime end boundaries.
In the Carath\'eodory prime end theory the only difference is
that the closed interval $\itoverline{I}$ corresponds to one prime end, apart
from this there is a one-to-one correspondence between the 
Carath\'eodory prime ends and the points in the extended boundary
$\bdyExt \comb$ (for this particular set).

In Adamowicz--Bj\"orn--Bj\"orn--Shan\-mu\-ga\-lin\-gam~\cite{ABBSprime}
a different definition of prime ends was proposed, which in the case of
the comb gives a natural one-to-one correspondence between its
prime end boundary $\bdyP \comb$  and the points in 
$\bdyExt \comb \setm I$,  whereas
there are no prime ends corresponding to points in $I$,
see Example~5.1 in \cite{ABBSprime}.
As we shall see this prime end boundary, and the associated topology,
will be of more interest in this paper than the  
Carath\'eodory prime end boundary.
For us it is enough to know that $\bdyP \comb= \bdyExt \comb \setm I$,
and we refer to \cite{ABBSprime} for their definition of prime ends.

If we introduce the \emph{Mazurkiewicz distance} 
(sometimes called \emph{inner diameter distance}) $\dM$ on $\comb$ by letting
\[
     \dM(x,y) =\inf \diam E,
\]
where the infimum is taken over all connected sets $E \subset \comb$
containing $x,y \in \comb$,
then $\bdyM \comb = \bdyP \comb$.
Here $\bdyM \comb=\clcombm \setm \comb$, 
where $\clcombm$ is the completion of $(\comb,\dM)$,
and the equality $\bdyM \comb = \bdyP \comb$
is understood in the sense that there is a homeomorphism
$H: \clcombm \to \clOmP$ such that $H|_{\comb}$ is the identity.
(For more on the Mazurkiewicz distance see \cite{ABBSprime} and 
Bj\"orn--Bj\"orn--Shanmugalingam~\cite{BBSdir},~\cite{BBStop}.)

Note that $\bdyExt \comb \setm I = \bdyP \comb = \bdyM \comb$ is \emph{not}
compact.

\section{Perron solutions}
\label{sect-Perron}

\begin{deff}
A function $u: \Om \to \R \cup \{\infty\}$ is \emph{\p-superharmonic} in
a domain (i.e.\ nonempty open connected set) $\Om$
if
\begin{enumerate}
\item $u$ is lower semicontinuous;
\item $u \not\equiv \infty$;
\item for each domain $G \Subset \Om$ and each $h \in C(\itoverline{G})$
which is \p-harmonic in $G$ and such that $h \le u$ on $\bdy G$
it is true that $h \le u$ in $G$.
\end{enumerate}
A function $v: \Om \to \R \cup \{-\infty\}$ is \emph{\p-subharmonic}
if $-v$ is \p-superharmonic.
\end{deff}

We will be interested in two types of Perron solutions.
We denote the standard Perron solutions using the letter $P$ and the
special ones using $S$.
For the latter we consider $\combExt$ with the prime end boundary
$\bdyP \comb:=\bdyExt \comb \setm I$.
The special Perron solutions are primarily used as a tool
in our study of the standard ones.

\begin{deff}   \label{def-Perron}
Given a function $f : \bdyExt \comb \to \eR$,
let $\UU_f$ be the set of all 
\p-superharmonic functions $u$ on $\comb$
bounded from below
such that 
\[ 
	\liminf_{\combExt \ni y \to  x} u(y) \ge f(x) \quad \text{for all }
	x \in \bdyExt \comb.
\] 
The \emph{extended upper Perron solution} of $f$ is the function
\[ 
    \EuHp f (x) = \inf_{u \in \UU_f}  u(x), \quad x \in \comb.
\]

Let similarly, for $f : \bdyP \comb  \to \eR$, 
$\UUt_f$ be the set of all 
\p-superharmonic functions $u$ on $\comb$
bounded from below
such that 
\[ 
	\liminf_{\combExt \ni y \to  x} u(y) \ge f(x) \quad \text{for all }
	x \in \bdyP \comb.
\] 
The  \emph{special upper Perron solution} of $f$ is the function
\[ 
    \SuHp f (x) = \inf_{u \in \UUt_f}  u(x), \quad x \in \comb.
\]

The lower Perron solutions are defined similarly using \p-subharmonic functions, 
or equivalently
by letting
\[
    \ElHp f = - \EuHp (-f)
    \quad \text{and} \quad
    \SlHp f = - \SuHp (-f).
\]
If $\EuHp f = \ElHp f$, then we let 
$\EHp f := \EuHp f$
and $f$ is said to be \emph{$\EHp$-resolutive}.
We similarly define $\SHp f$ and \emph{$\SHp$-resolutivity}.

We also similarly define $\uHp f$, $\lHp f$ and $\Hp f$
for $f : \bdy \comb \to \eR$,
and  $\uHpind{\Om} f$,  $\lHpind{\Om} f$ and $\Hpind{\Om} f$
for $f: \bdy \Om \to \eR$ for bounded domains $\Om$.
\end{deff}

The proof that standard Perron solutions are \p-harmonic or
identically $\pm \infty$ directly carries over to our special
Perron solutions, see Theorem~9.2 in 
Heinonen--Kilpel\"ainen--Martio~\cite{HeKiMa}.

The following comparison principle shows
that $\SlHp f \le \SuHp f$ and $\ElHp f \le \EuHp f$ for all functions $f$.
Since it is immediate that $\ElHp f \le \SlHp f$ and $\SuHp f \le \EuHp f$, 
we find that
\begin{equation} \label{eq-fund}
        \ElHp f \le \SlHp f \le \SuHp f \le \EuHp f.
\end{equation}
Moreover, if $f : \bdy \comb \to \eR$, then $f$
can naturally
be seen as a function on $\bdyExt \comb$, and we will do so 
without further ado.
It is easy to see that in this case we always have
$\lHp f= \ElHp f$ and
$\uHp f= \EuHp f$.

Another obvious fact is that
if $f_1 \le f_2$, then $\EuHp f_1 \le \EuHp f_2$ and similar
inequalities also follow for all the 
other (lower and upper) types of Perron solutions.
We will say that this inequality holds by \emph{simple comparison}.
This should be seen in relation to the following important
comparison principle.

\begin{thm} \label{thm-comp}
\textup{(Extended comparison principle)}
Assume that $u$ is \p-superharmonic and
$v$ is \p-subharmonic in\/ $\comb$.
If
\begin{equation} \label{eq-comp-1}
         \liminf_{\combExt \ni y \to x} (u(y)-v(y)) \ge 0
         \quad \text{for all } x \in \bdyP \comb,
\end{equation} 
which in particular holds if 
\begin{equation} \label{eq-comp}
       \infty \ne  \limsup_{\combExt \ni y \to x} v(y)
        \le \liminf_{\combExt \ni y \to x} u(y) \ne -\infty
        \quad \text{for all } x \in \bdyP \comb,
\end{equation} 
then $v \le u$ in\/ $\comb$.
\end{thm}

This can be achieved by a modification of the proof
of the corresponding result for standard Perron solutions, 
see Theorem~3.1 in Bj\"orn~\cite{ABjump}.
We here instead use the comparison principle on $\combExt_k$
as a tool, to give a shorter proof,
where
$\comb_k:=\comb \setm ([0,1) \times (0,2^{-k}))$ 
equipped with the Mazurkiewicz distance.
The comparison principle on $\combExt_k$ is given in
Proposition~7.2 in Bj\"orn--Bj\"orn--Shanmugalingam~\cite{BBSdir} % Check no
under the assumption \eqref{eq-comp}.
The proof therein however first deduces \eqref{eq-comp-1} and then
proceeds from this assumption.

\begin{proof}
Let $x_0 \in \comb$ and $\eps>0$. 
Then there is $k$ such that $x_0 \in \comb_k$
and such that $v(y) \le u(y)+\eps$ if $|y|<2^{-k}$.
Hence
\[
        \liminf_{\combExt_k \ni y \to x} (u(y)+\eps -v(y))  \ge 0
        \quad \text{for all } x \in \bdyExt \comb_k.
\]
By the comparison principle for $\combExt_k$,
see the proof of Proposition~7.2 in \cite{BBSdir}, % Check no
we get that
$v \le u+\eps$ in $\comb_k$, and in particular
$v(x_0)  \le u(x_0)+\eps$.
Letting $\eps \to 0$ completes the proof.
\end{proof}

\section{Boundary regularity}
\label{sect-bdy-reg}

The prime end boundary $\bdyP \comb$ is not compact, and thus we have to
take extra care when defining boundary regularity. 
The three classes $C(\bdyP \comb)$, 
$\Cbdd(\bdyP \comb)$ (of bounded continuous functions)
and
$\Cunif(\bdyP \comb)$ (of uniformly continuous functions)
do not coincide as they do on compact sets.
For the results in this paper it seems that $\Cbdd(\bdyP \comb)$ 
is the right choice in the following definition of boundary regularity.

\begin{deff}
A point $x_0 \in \bdyP \comb$ is \emph{$S$-regular} if
\[
       \lim_{\combExt \ni y \to x_0} \SuHp f(y) = f(x_0)
       \quad \text{for all }f \in \Cbdd(\bdyP \comb).
\]
\end{deff}

That we cannot allow for general $f \in C(\bdyP \comb)$ is due to the fact
that there are $f,h \in C(\bdyP \comb)$ such that $\SHp f \equiv \SuHp h \equiv \infty$
and $\SlHp h \equiv -\infty$, 
as shown by the following example.

\begin{example} \label{ex-1}
Fix $x_0 \in \comb$ and
let  $f_j(x)=(1-2^{j+2}|x-y_j|)_\limplus$, $j=1,2,\ldots$,  where
$y_j=(1,3 \cdot 2^{-j-1})$.
By the $S$-regularity of $y_j$ (shown 
in Proposition~\ref{prop-Sreg}
below)
we see that $\SlHp f_j \not \equiv 0$.
Thus the strong minimum principle,
see Theorem~7.12 in  
Heinonen--Kilpel\"ainen--Martio~\cite{HeKiMa},
yields $\SlHp f_j(x_0)>0$.
Let 
\[
     f=\sum_{j=1}^\infty \frac{2j f_{2j}}{\SlHp f_{2j}(x_0)} \in C(\bdyP \comb)
     \quad \text{and} \quad
     h=\sum_{j=1}^\infty \frac{(-1)^jj f_j}{\SlHp f_j(x_0)} \in C(\bdyP \comb).
\]
Then $\SlHp f(x_0) \ge j$ for all $j$, and thus
$\SlHp f(x_0)=\infty$.
It follows that $\SHp f \equiv \infty$.

Moreover, if $u \in \UUt_h$ then, by definition,  
$u \ge -m$ for some real $m \ge 0$.
Hence $0 \le u+m \in \UUt_f$, but this contradicts the fact that
$\SuHp f \equiv \infty$.
Thus there is no such $u$, i.e.\ $\SuHp h \equiv \infty$.
Similarly $\SlHp h \equiv -\infty$.

By \eqref{eq-fund} it follows that $\uHp h \equiv \EuHp h \equiv \infty$
and $\lHp h \equiv \ElHp h \equiv -\infty$.
Hence the resolutivity in Theorem~\ref{thm-intro} is not true
for arbitrary unbounded continuous functions 
on $\bdy \comb \setm I$.
\end{example} 

\begin{prop} \label{prop-Sreg}
Let  $x_0 \in \bdyP \comb$.
Then $x_0$ is $S$-regular.
\end{prop}

\begin{proof}
We will use that all boundary points of $\bdy \comb$ are
regular (with respect to the standard nonextended Perron solutions),
which is well-known and e.g.\ follows from the sufficiency part of the
Wiener criterion, see Maz{\cprime}ya~\cite{mazya70}.
To do so we need to distinguish those points $x_0$ 
for which $\Phi(x_0)$ has a unique
preimage in $\bdyP \comb$ and those which have two preimages.
(Recall that $\Phi:\bdyExt \comb \to \bdy \comb$ is the natural map.)

\medskip

\emph{Case} 1. \emph{$\Phi(x_0)$ has the unique preimage $x_0$.}
Let $f \in \Cbdd(\bdyP \comb)$.
Then we can find $h \in C(\bdy \comb)$ such that $h(x_0)=f(x_0)$
and $h \ge f$ on $\bdyP \comb$.
By simple comparison, 
we have $\SuHp f \le \SuHp h \le \EHp h = \Hp h$ in $\comb$.
Thus, using that all boundary points of $\comb$ are regular for
the standard Perron solutions, we see that
\[
       \limsup_{\combExt \ni y \to x_0} \SuHp f(y) 
       \le \lim_{\comb \ni y \to x_0} \Hp h(y) 
      = h(x_0)=f(x_0).
\]
Similarly,
$
       \liminf_{\combExt \ni y \to x_0} \SlHp f(y) 
       \ge f(x_0)$,
which together with the inequality $\SlHp f\le \SuHp f$
shows that 
$
       \lim_{\combExt \ni y \to x_0} \SuHp f(y) 
       = f(x_0)$.
As $f$ was arbitrary this yields the $S$-regularity of $x_0$.

\medskip

\emph{Case} 2. \emph{$\Phi(x_0)$ has two preimages\/ \textup{(}of
which $x_0$ is one\/\textup{)}.}
In this case $x_0 \in I_j$ for some $j$ and moreover
$x_0 \in \bdy G$, where $G=(2^{-k},2^{1-k}) \times (0,1)$ for some $k$.
Let $f \in \Cbdd(\bdyP \comb)$
and assume that $0 \le f \le 1$.
Then we can find $h \in C(\bdy G)$ such that $h(x_0)=f(x_0)$,
$f \le h \le 1$ on $\bdy G \cap \bdyP \comb $ 
and $h=1$ on $\bdy G \setm \bdyP \comb $.

Let $u$ be a \p-superharmonic function competing in the definition
of $\uHpind{G} h$, and set
\[
    v = \begin{cases}
         \min\{u,1\} & \text{in } G, \\
         1 & \text{in } \comb \setm G.
      \end{cases}
\]
Then $v$ is \p-superharmonic in $\comb$, by  Pasting lemma~7.9 in
Heinonen--Kilpel\"ainen--Martio~\cite{HeKiMa},
and thus $v \in \UUt_f$.
Hence $\SuHp f \le u$ in $G$, and since $u$ was arbitrary,
$\SuHp f \le \Hpind{G} h$ in $G$.
Thus, using also that all boundary points of $G$ are regular for
the standard Perron solutions, we see that
\[
       \limsup_{\combExt \ni y \to x_0} \SuHp f(y) 
       \le \lim_{G \ni y \to x_0} \Hpind{G} h(y) 
      = h(x_0)=f(x_0).
\]
Continuing exactly as in case~1 we deduce 
the $S$-regularity of $x_0$.
\end{proof}

We can now use this to deduce $S$-resolutivity for $f \in \Cbdd(\bdyP \comb)$.

\begin{prop} \label{prop-S-resolutive}
Let $f \in \Cbdd(\bdyP \comb)$.
Then $f$ is $S$-resolutive.
\end{prop}

\begin{proof}
By the regularity of $x \in \bdyP \comb$ we see that
\[
      \lim_{\combExt \ni y \to x} \SuHp f(y) = f(x)
      = \lim_{\combExt \ni y \to x} \SlHp f(y)
        \quad \text{for all } x \in \bdyP \comb.
\]
As $\SuHp f$ and $\SlHp f$ are \p-harmonic we can apply the
comparison principle (Theorem~\ref{thm-comp}), with $\SuHp f$ as the \p-subharmonic
function and $\SlHp f$ as the \p-superharmonic function, to deduce that
$\SuHp f \le \SlHp f$.
Since we always have $\SlHp f \le \SuHp f$, we see that $\SlHp f = \SuHp f$.
\end{proof}

\begin{prop} \label{prop-Extreg}
Let $x_0 \in \bdyExt \comb$. 
Then $x_0$ is \emph{$\EHp$-regular}, i.e.\
\[
      \lim_{\combExt \ni y \to x_0} \EuHp f(y) = f(x_0)
      \quad \text{for all } f \in C(\bdyExt \comb).
\]
\end{prop}

The proof of this result is almost identical to the proof
of Proposition~\ref{prop-Sreg} above,
and we leave it to the interested reader to verify.
As in Proposition~\ref{prop-S-resolutive} this can be used
together with the comparison principle (Theorem~\ref{thm-comp})
to obtain the
$\EHp$-resolutivity for all $f \in C(\bdyExt \comb)$,
which however is merely a special case of our main result
(Theorem~\ref{thm-main}) below.

We will need the following consequence of 
Proposition~\ref{prop-Extreg}.

\begin{prop} \label{prop-Extreg-gen}
Let $x_0 \in \bdyExt \comb$
and let $f : \bdyExt \comb \to \R$ be a function which
is lower semicontinuous at $x_0$ and
 bounded  on $\bdyExt \comb$.
Then
\[ 
      \lim_{\combExt \ni y \to x_0} \EuHp f(y) 
      \ge  \lim_{\combExt \ni y \to x_0} \ElHp f(y) \ge f(x_0).
\] 
If $f$ is moreover continuous at $x_0$, then 
\begin{equation} \label{eq-Extreg-gen-cont}
      \lim_{\combExt \ni y \to x_0} \EuHp f(y) 
      = \lim_{\combExt \ni y \to x_0} \ElHp f(y) = f(x_0).
\end{equation}
\end{prop}

\begin{proof}
(The proof  is similar to 
the corresponding result for standard Perron solutions,
see Proposition~7.1 in 
Bj\"orn--Bj\"orn--Shanmugalingam~\cite{BBS2}
or Theorem~10.29 in \cite{BBbook}.)
We can find a function $h \in C(\bdyExt \comb)$ such that
$h \le f$ on $\bdyExt \comb$ and $h(x_0)=f(x_0)$.
By simple comparison and the Ext-regularity obtained in
Proposition~\ref{prop-Extreg} we get that
\[
      \lim_{\combExt \ni y \to x_0} \EuHp f(y) 
      \ge \lim_{\combExt \ni y \to x_0} \ElHp f(y) 
      \ge \lim_{\combExt \ni y \to x_0} \EHp h(y) = h(x_0)=f(x_0).
\]
If $f$ is continuous at $x_0$ we apply this also to $-f$ to obtain
\eqref{eq-Extreg-gen-cont}.
\end{proof}

\section{The main result}
\label{sect-main}

The following is the main result of this paper,
and Theorem~\ref{thm-intro} is a special case of this result
since $\uHp f= \EuHp f$ if $f : \bdy \comb \to \eR$.

\begin{thm} \label{thm-main}
Let $f : \bdyExt \comb \to \eR$ be such that $f|_{\bdyP \comb} \in \Cbdd(\bdyP \comb)$.
Then
\[
      \EHp f = \SHp f.
\]
In particular, 
$f$ is $\EHp$-resolutive and $\EHp f$ is independent of
$f|_I$,
i.e.\ 
if $h : \bdyExt \comb \to \eR$ is such that
$h=f$ on $\bdyP \comb$, then
$
      \EHp h = \EHp f = \SHp h  = \SHp f.
$
\end{thm}

In the special case when $f$ is bounded on $\bdyExt \comb$
and continuous at $0$ also from $I$, 
this can be deduced directly from 
the comparison principle (Theorem~\ref{thm-comp})
and Proposition~\ref{prop-Extreg-gen}.
Hence it is to allow for a discontinuity at $0$ that we need to 
work harder.

To prove this we will use a number of results which are
available to us for the comb, but not in more general situations. 
Let us mention the key ingredients which are not generally available,
but first we need  some more terminology.

A boundary point $x_0 \in \bdy \Om$ is \emph{semiregular} (with respect to a 
domain $\Om$)
if $x_0$ is irregular and the limit
\[
      \lim_{\Om \ni y \to x_0} \Hpind{\Om} f(y)
      \quad \text{exists for all } f \in C(\bdy \Om).
\]
An open set is \emph{semiregular} if all its boundary
points are either regular or semiregular.
There are two types of irregular boundary points,
semiregular and strongly irregular boundary points,
with very different behaviour,
see Bj\"orn~\cite{ABclass}.

\begin{enumerate}
\item
We will need the comparison principle (Theorem~\ref{thm-comp})
on $\bdyP \comb$. 
It needs further investigation to see which sets this can be extended to,
see Estep--Shanmugalingam~\cite{ES}.
\item 
We will use that all boundary points are regular. 
However, if there are also some semiregular boundary points
it should be possible to combine the techniques 
for proving the comparison principles in 
Theorem~3.1 in Bj\"orn~\cite{ABjump} and in Theorem~\ref{thm-comp}
to obtain a suitable comparison principle enabling
the proof of Theorem~\ref{thm-main} in such a case.

The situation resembles the one when proving that
bounded semicontinuous functions are resolutive, 
see the discussion in the introduction,
although for our new result semiregularity should be possible to handle, whereas
strong irregularity is still a serious obstacle.
\item
For $p>2$ we will also need Theorem~6.3 in \cite{ABjump},
a result which is only available in unweighted $\R^n$.
Here we apply it for the point $0$ and to be able to do so we
need to know that $0$ is an exterior ray point.
\item
In the unbounded case (i.e.\ when $f$ is allowed to be unbounded on $I$)
we will also need a recent result from
Bj\"orn--Bj\"orn--Shanmugalingam~\cite{BBSdir}.
\end{enumerate}

As there are some  extra complications
to obtain this result in the unbounded case
we first give a proof for the bounded case.
It should also be said that
Theorem~\ref{thm-main} is a special case of
Theorem~\ref{thm-comb-last} below
(the role of  $f$ in Theorem~\ref{thm-main} is taken by
$h$ in Theorem~\ref{thm-comb-last}).
However, as the proof of Theorem~\ref{thm-comb-last}
is substantially more involved we prefer to give a 
direct proof of Theorem~\ref{thm-main} here.

\begin{proof}(For bounded $f$.)
Assume, without loss of generality, 
that $0 \le f \le 2$ and that $f(0)=1$.
Let 
\[
    k=\begin{cases}
        f & \text{on } \bdyP \comb, \\
        1 & \text{on } I,
      \end{cases}
      \quad \text{and} \quad
    \ktilde=k+\chi_{\{0\}}.
\]
Let also $u \in \UUt_{\ktilde}$.
Then
\[
      \liminf_{\comb \ni y \to 0} u(y) \ge \ktilde(0)=2
        \ge \limsup_{\comb \ni y \to 0} \EuHp f(y).
\]
Since $f \in \Cbdd(\bdyP \comb)$,  Proposition~\ref{prop-Extreg-gen}
shows that
\[
      \lim_{\combExt \ni y \to x} \EuHp f(y)
     =f(x)
    \le \liminf_{\combExt \ni y \to x} u(y)
     \quad \text{for all } x \in \bdyP \comb \setm \{0\}.
\]
Thus, the comparison principle (Theorem~\ref{thm-comp}) yields that
$u \ge \EuHp f$.
As this is true for all  $u \in \UUt_{\ktilde}$, we obtain 
that $\EuHp \ktilde \ge \EuHp f$.

Next, 
since $k$ is continuous at $0$, 
we can find $\psi \in C(\bdy \comb)$ such that
$\psi(0)=1$ and $k \le \psi$ on $\bdyExt \comb$.
(Note that $\psi$ is a function on $\bdy \comb$.)
Let also $\psit=\psi+\chi_{\{0\}}$.
Then $\Hp \psit = \Hp \psi$ by 
either Theorem~6.3 in Bj\"orn~\cite{ABjump} (if $p>2$)
or 
Theorem~6.1 in Bj\"orn--Bj\"orn--Shanmugalingam~\cite{BBS2} (if $p \le 2$)
(which can also be found as Theorem~10.29 in \cite{BBbook};
the more general Theorem~9.1 in 
Bj\"orn--Bj\"orn--Shanmugalingam~\cite{BBSdir} can also be used). % Check no
By simple comparison,
\[
     \EuHp \kt \le \Hp \psit = \Hp \psi.
\]
Thus using also regularity (for $\psi$) 
and $S$-regularity (for $f$), see Proposition~\ref{prop-Sreg}, we see that
\[
      \limsup_{\comb \ni y \to 0} \EuHp \ktilde(y) 
      \le \lim_{\comb \ni y \to 0} \Hp \psi(y) = 1
      = \lim_{\comb \ni y \to 0} \SHp f(y).
\]
Moreover, by  Proposition~\ref{prop-Extreg-gen} and
the $S$-regularity again
we get that
\[
      \lim_{\combExt \ni y \to x} \EuHp \ktilde(y)
     =f(x)
     =       \lim_{\combExt \ni y \to x} \SHp f(y)
     \quad \text{for all } x \in \bdyP \comb \setm \{0\}.
\]
Using the comparison principle (Theorem~\ref{thm-comp}) 
we obtain that $\EuHp f \le \EuHp \ktilde \le \SHp f$.
Applying this also to $-f$ yields
\[
   \SHp f = - \SHp (-f) \le \ElHp f \le \EuHp f \le \SHp f.
   \qedhere
\]
\end{proof}

\begin{proof}(The general case.)
Assume, without loss of generality, 
that $0 \le f \le 2$ on $\bdyP \comb$ and that $f(0)=1$.
Let 
\[
    k=\begin{cases}
        f & \text{on } \bdyP \comb, \\
        1 & \text{on } I,
      \end{cases}
      \quad \text{and} \quad
    \ktilde=k+\chi_{\{0\}}.
\]
Let also $u \in \UUt_{\ktilde}$.
We want to show that $u \ge \EuHp f$.
To do so is a fair bit more involved in the unbounded case than in
the bounded case. 
However, once this has been achieved we can proceed exactly as in the 
bounded case.

Fix $x \in \bdyP \comb \setm \{0\}$ for the moment.
Let $m$ be a positive integer such that $2^{1-m} < |x|$.
Let also $G=((-1,1)  \times (0,2)) \setm \bigcup_{j=0}^m \itoverline{I}_j$,
and let $d_{G}$ be the Mazurkiewicz distance with respect to $G$. 
We equip $\comb$ with the distance $d_{G}$ and call this
space $\combG=(\comb,d_{G})$.
Taking the completion of this space we obtain $\bdy_{\itoverline{G}^M} \comb$, in a similar
way as when we obtained $\bdyExt \comb$.
(This time only the points in $I_j$, $j=0,\ldots,m$, are doubled,
whereas to each point in $\bigcup_{j=m+1}^\infty I_j$ there is just one
corresponding point in $\bdy_{\itoverline{G}^M} \comb$.
The notation follows Bj\"orn--Bj\"orn--Shanmugalingam~\cite{BBSdir}.)

Next we find $\phi \in C(\bdy_{\itoverline{G}^M} \comb)$ such that
$k \le \phi \le 2$ and $\phi(x)=k(x)=f(x)$.
Let also $\phit = \phi + \infty \chi_I$.
By Theorem~11.2  % Check no
in \cite{BBSdir} together with \eqref{eq-bCpI},
we see that $P_{\combG} \phit = P_{\combG} \phi$.
Thus, by simple comparison,  we obtain that
\begin{equation} \label{eq-EuHpf}
    \EuHp f \le \EuHp \phit = P_{\combG} \phit = P_{\combG} \phi  \le 2.
\end{equation}
That $x$ is a regular boundary point with respect to $\combG$ is shown
as in Proposition~\ref{prop-Sreg}. Using this we find that
\[
      \lim_{\combExt \ni y \to x} \EuHp f(y)
     \le \lim_{\combG \ni y \to x} P_{\combG} \phi(y)
     = \phi(x) =f(x)
    \le \liminf_{\combExt \ni y \to x} u(y).
\]
Since $x \in \bdyP \comb \setm \{0\}$ was arbitrary we have thus  shown that
\[
      \lim_{\combExt \ni y \to x} \EuHp f(y)
    \le \liminf_{\combExt \ni y \to x} u(y)
    \quad \text{for all } x \in \bdyP \comb \setm \{0\}.
\]
Moreover, using \eqref{eq-EuHpf} again, we see that
\[
      \lim_{\comb \ni y \to 0} \EuHp f(y)
      \le 2 
      \le  \liminf_{\comb \ni y \to 0} u(y).
\]
Hence, the comparison principle (Theorem~\ref{thm-comp}) yields that
$u \ge \EuHp f$.
As already mentioned, the rest of the proof is exactly as in the bounded case.
\end{proof}

\section{Functions with jumps}
\label{sect-jumps}

In this section we go one step further and combine the technique above with 
the technique 
in Bj\"orn~\cite{ABjump} to deduce the following result.

\begin{thm} \label{thm-comb-last}
Let $E \subset \bdyP \comb \setm \{0\}$ be a countable set.
Assume that $f : \bdyExt \comb \to \R$ is bounded and that $f|_{\bdyP \comb}$
is continuous at all points in $\bdyP \comb  \setm E$ and has jumps at all
points in $E$.

Let  $h: \bdyExt \comb \to \eR$ be such that
$h=f$ on $\bdyP \comb \setm \Et$, where $\Et \subset \bdyP \comb$ and
\begin{equation}  \label{eq-CpEt}
     \begin{cases}
     \Cp(\Phi(\Et))=0, & \text{if\/ } 1 <p \le 2, \\
     \Et \text{ is countable}, & \text{if } p >2.
     \end{cases}
\end{equation}
Then both $f$ and $h$ are $\EHp$- and $\SHp$-resolutive, and
\[
               \EHp h = \EHp f = \SHp h = \SHp f.
\]
\end{thm}

Here $\Cp$ is the Sobolev capacity on $\R^2$, see
p.\ 48 in Heinonen--Kilpel\"ainen--Martio~\cite{HeKiMa}.
(Recall also that $\Phi: \bdyExt \comb \to \bdy \comb$ is the natural map.)
By saying that $f$ has a jump at $x \in \bdyP \setm \{0\}$ 
we mean that it has limits
from the two directions along the boundary, but
these limits need not be the same, neither do we impose
any condition on the relation between these limits and the value $f(x)$.

\begin{remark} \label{rmk-mods}
We need to use modifications of Theorems~5.2 and~5.4 in Bj\"orn~\cite{ABjump}
for the $\EHp$-Perron solutions,
and the proofs therein  directly generalize to this
situation.
We also need to use Theorem~5.2 in \cite{ABjump}
for jumps at the tips
of the comb's teeth, where the 
angle is $2 \pi$. 
Indeed, in the proof therein
we should see $\Omt$ as a Riemann surface, on whose closure
we consider Perron solutions (the theory being the same to the small
extent used in the proof).
When applying Lemma~7.28 in Heinonen--Kilpel\"ainen--Martio~\cite{HeKiMa},
it is easy to deduce that $\vt$ is \p-subharmonic in $\Om$ as well as in
$\Omt$, which is enough for the rest of the proof.
(It also follows that the other results 
in \cite{ABjump} are true also for asymptotic corner points with
angle $2\pi$.)
\end{remark}

For $p>2$ we also need the following key lemma.
(For simplicity we use some obvious complex notation.)

\begin{lem} \label{lem-key} 
Assume that $p>2$ and that $f$, $h$, $E$ and $\Et$ are as in 
Theorem~\ref{thm-comb-last}.
Let $x_0 \in \bdyP \comb$.

If $x_0 \ne 0$, then 
we let 
\[
    U(x_0+\reith)=A_1+(A_2-A_1)\frac{\theta-\alp_1}{\alp_2-\alp_1}
    \quad \text{for } r>0 \text{ and }\alp_1 < \theta < \alp_2,
\]
where  $\alp_1<\alp_2 \le \alp_1 + 2\pi$ are the two
directions of $\bdyP \comb$ near $x_0$ chosen so that
$U$ is defined in a neighbourhood of $x_0$ in $\combExt$,
and $A_j=\lim_{t \to 0\limplus} f(x_0+te^{i\alp_j})$, $j=1,2$.

If $x_0=0$
\textup{(}when we do not just have two directions\/\textup{)}
 we instead let $A_1=A_2=U(x)=f(x_0)$ for all $x \in \R^2$.

Then
\begin{equation} \label{eq-lim-gen-p>2}
    \lim_{\combExt\ni z \to x_0} {(\ElHp h(z) - U(z))}=
    \lim_{\combExt \ni z \to x_0} {(\EuHp h(z) - U(z))}=0.
\end{equation}
\end{lem}

If $x_0$ is a tip point we have $\alp_2=\alp_1+2\pi$ and we
should interpret the statement above using a Riemann surface
as in Remark~\ref{rmk-mods}.

Note that in general it is not known
if $\EuHp k = \lim_{m \to \infty} \EuHp \min\{k,m\}$,
which makes it necessary to use induction in the proof below
even in the case when $\Et$ is just one point.

\begin{proof}
Without loss of generality assume that 
$A_1=0 \le A_2$ and that $x_0 \in \Et = E \cup \{0\}$.
We can find a nonnegative bounded function $k: \bdyExt \comb \to \R$ such that
\begin{enumerate}
\item
$k \ge f$ on $\bdyP \comb$;
\item
$k(x)=k(0)$ for $x \in I$;
\item 
$k$ is continuous at all points in $\bdyP \comb \setm \{x_0\}$;
\item 
$k$ is lower semicontinuous at all points in $I$;
\item
if $x_0 \ne 0$, then 
$k$ has a jump at $x_0$ with limits $0$ and $A_2$
and $k(x_0) = \sup_{\bdyP \comb} k$;
\item
while if $x_0 = 0$, 
we require that 
$
     k(0)=\lim_{\bdyP \comb \ni y \to 0} k(y)=0.
$
\end{enumerate}
Note in particular that $k$ is upper semicontinuous at $x_0$.
If  $x_0=0$, then $k$ is even continuous at $x_0$
(but it need not be continuous at the points in $I$).

Let $z_0 \in \comb$ and $\eps>0$.
Let also $\{y_j\}_{j=0}^\infty$ be a sequence of
points in $\Et$ such that each point in $\Et$
appears infinitely many times.
We want to construct an increasing  sequence $\{k_j\}_{j=0}^\infty$
of bounded functions on $\bdyExt \comb$ such that $k_0=k$ and 
for each nonnegative integer $j$,
\begin{enumerate}
\renewcommand{\theenumi}{\textup{(\roman{enumi})}}%
\item 
$k_{j+1} - k_j \in C(\bdyExt \comb)$;
\item 
$ k_j \le k_{j+1} \le k_j +1$;
\item
$\EuHp k_{j+1}(z_0) \le \EuHp k_{j}(z_0) +2^{-j} \eps$;
\item 
$k_{j+1}(y_j)=k_j(y_j)+1$;
\item
$k_{j+1}(x)=k_{j+1}(0)$ for $x \in I$.
\end{enumerate}

We proceed by induction and assume that $k_j$ has been constructed
for some nonnegative integer $j$.
(The initial step is of course to let $k_0=k$.)
Let 
\[
	\kt_j =k_j + 2\chi_{E_j},
        \quad \text{where }
        E_j=\begin{cases}
          \{y_j\}, & \text{if } y_j \ne 0, \\
          \itoverline{I}, & \text{if } y_j = 0.
          \end{cases}
\]
We want to use the 
comparison principle (Theorem~\ref{thm-comp}) to 
show that $\EuHp \kt_j = \EuHp k_j$.
To do so we need to establish that
\begin{equation} 
  \label{eq-kt-k}
\lim_{\combExt \ni y \to x} {(\EuHp \kt_j(y) - \EuHp k_j(y))} =0
\quad \text{for all } x \in \bdyP \comb.
\end{equation}

If $x_0 \ne 0$, then
\begin{align}
    & \lim_{\combExt \ni y \to x_0} {(\EuHp \kt_j(y) - \EuHp k_j(y))} 
    \label{eq-ktj-1}
\\
    & \kern 2em =
    \lim_{\combExt \ni y \to x_0} {(\EuHp \kt_j(y) - U_j(y))}
   -    \lim_{\combExt \ni y \to x_0} {(\EuHp k_j(y) - U_j(y))}
   =0,\kern 3em \nonumber
\end{align}
by Theorem~5.2 in Bj\"orn~\cite{ABjump} and
Remark~\ref{rmk-mods}
(applied to both $\kt_j$ and $k_j$),
where $U_j$ is the function called $U$ in Theorem~5.2 in \cite{ABjump}
(translated to $x_0$).
Note that the same function $U_j$ applies to both $\kt_j$ and $k_j$.

Theorem~5.4 in \cite{ABjump} 
and Remark~\ref{rmk-mods}
(applied to both $\kt_j$ and $k_j$)
yield
\[
       \lim_{\combExt \ni y \to y_j} \EuHp \kt_j(y) 
       = k_j(y_j) =\lim_{\combExt \ni y \to y_j} \EuHp k_j(y),
	\quad \text{if } y_j \notin \{x_0,0\}.
\]
Moreover, 
if $x \in \bdyP \comb \setm \{x_0,y_j\}$ or if $x=x_0=0 \ne y_j$,
then
$k_j$ and $\kt_j$ are continuous at $x$, and thus,
by Proposition~\ref{prop-Extreg-gen},
\begin{equation}
    \label{eq-ktj-2}
       \lim_{\combExt \ni y \to x} \EuHp \kt_j(y) 
	= \kt_j(x)= k_j(x)
       =\lim_{\combExt \ni y \to x} \EuHp k_j(y). 
\end{equation}

It remains to handle the case when $x=y_j=0$
for which we will use the auxiliary function 
$k'_j =k_j + 2\chi_{\{0\}}$. 
Let $u' \in \UU_{k'_j}$. 
Then 
\[
\liminf_{\combExt \ni y \to x} (u'(y) - \EuHp \kt_j(y)) 
  \ge \lim_{\combExt \ni y \to x} {(\EuHp k'_j(y) - \EuHp \kt_j(y))} =0
\]
for all $x \in \bdyP \comb \setm \{0\}$,
where
the equality is obtained as in \eqref{eq-ktj-1} and \eqref{eq-ktj-2}.
Also
\[
     \liminf_{\comb \ni y \to 0} u'(y) \ge k'_j(0),
\]
while 
\[
     \limsup_{\comb \ni y \to 0} \EuHp \kt_j(y) \le \kt_j(0) = k'_j(0),
\]
by Proposition~\ref{prop-Extreg-gen} (applied to $-\kt_j$)
since $\kt_j$ is upper semicontinuous at $0$.
By the comparison principle (Theorem~\ref{thm-comp}), we
see that 
$u' \ge \EuHp \kt_j$,
and since this holds for all $u'  \in \UU_{k'_j}$,
we obtain that $\EuHp k'_j \ge \EuHp \kt_j$.
The converse inequality holds by simple comparison, and hence
$\EuHp k'_j = \EuHp \kt_j$.
Theorem~5.4 in \cite{ABjump} 
and Remark~\ref{rmk-mods} again
(this time applied to $k'_j$ and $k_j$)
yield
\[
       \lim_{\comb \ni y \to 0} \EuHp \kt_j(y) 
   =  \lim_{\comb \ni y \to 0} \EuHp k'_j(y) 
       = k_j(0) =\lim_{\comb \ni y \to 0} \EuHp k_j(y),
\]
which finally shows 
\eqref{eq-kt-k} for all $x \in \bdy_p \comb$ regardless
of the values of $x_0$ and $y_j$.
We thus conclude that 
$\EuHp \kt_j \equiv \EuHp k_j$, 
by 
the comparison principle (Theorem~\ref{thm-comp}).

Therefore, we can find $u \in \UU_{\kt_j}$ such that
\[
    u(z_0) < \EuHp \kt_j(z_0) +\frac{\eps}{2^j}
    = \EuHp k_j(z_0) +\frac{\eps}{2^j}.
\]
Extend $u$ to $\bdyExt \comb$ by letting
\[
     u(x)=\liminf_{\combExt \ni y \to x} u(y),
	\quad x \in \bdyExt \comb.
\]
Then $u$ is lower semicontinuous on $\overline{\combExt}$
and $u \ge \kt_j$ on $\bdyExt \comb$.

As 
$u$ is lower semicontinuous,
$k_j$  upper semicontinuous,
$u\ge \kt_j=k_j+ 2\chi_{E_j}$,
and $E_j$ is  compact,
there is $r>0$
such that
\[
          u(x) > k_j(x)+1 
	\quad \text{if } x \in \bdyExt \comb
        \text{ and } \dist_{\rm Ext}(x,E_j) <  r.
\]
If $y_j \ne 0$, then we moreover require that $r < \dist_{\rm Ext}(y_j,I)$.
Let 
\[
    k_{j+1}(x)=k_j(x) +\biggl(1-\frac{\dist_{\rm Ext}(x,E_j)}{r}\biggr)_\limplus,
		\quad x \in \bdyExt \comb.
\]
Then $u  \ge k_{j+1}$ on $\bdyExt \comb$.
Hence $u \in \UU_{k_{j+1}}$ and
\[
	\EuHp k_{j+1}(z_0) \le u(z_0) < 
            \EuHp k_j(z_0) +\frac{\eps}{2^j}.
\]
That the other requirements on $k_{j+1}$ are fulfilled is clear.
We have therefore completed the construction of the sequence
$\{k_j\}_{j=0}^\infty$.

It follows directly that $\{\EuHp k_j\}_{j=0}^\infty$
is an increasing sequence of \p-harmonic functions in $\comb$.
Let $v=\lim_{j \to \infty} \EuHp k_j$.
Since
\[
    \EuHp k_j(z_0) < \EuHp k(z_0) + \eps \sum_{k=0}^{j-1} 2^{-j}
	         < \EuHp k(z_0) + 2\eps,
\]
we see that 
$
    v(z_0) \le \EuHp k(z_0) + 2\eps < \infty.
$
Harnack's convergence theorem
(see Theorem~6.14 in
Heinonen--\allowbreak Kilpel\"ainen--Martio~\cite{HeKiMa})
shows that $v$ is \p-harmonic in $\comb$.
We next want to show that $v \in \UU_{h}$.
For $x \in \bdyP \comb \setm \Et$
we have, by Proposition~\ref{prop-Extreg-gen}, that
\[
    \liminf_{\combExt \ni y \to x} v(y)   
    \ge \liminf_{\combExt \ni y \to x} \EuHp k(y)
    = k(x) \ge f(x)=h(x).   
\]
On the other hand, if $x \in I \cup \Et  \setm \{x_0\}$,
then $k_j$ is lower semicontinuous at $x$, 
and thus, by Proposition~\ref{prop-Extreg-gen},
\[
    \liminf_{\combExt \ni y \to x} v(y)   
    \ge \lim_{j \to \infty} \liminf_{\combExt \ni y \to x} \EuHp k_j(y)
    \ge  \lim_{j \to \infty}  k_j(x)
    = \infty.   
\]
Since $x_0$ is $\Ext$-regular, by Proposition~\ref{prop-Extreg},
and $k_j -k \in C(\bdyExt \comb)$
we see that
\[
    \liminf_{\combExt \ni y \to x_0} \EuHp k_j(y)
         \ge  \liminf_{\combExt \ni y \to x_0} \EuHp (k_j-k)(y) 
    =   (k_j(x_0)-k(x_0)).
\]
Hence
\begin{align*}
    \liminf_{\combExt \ni y \to x_0} v(y)   
     \ge \lim_{j \to \infty} \liminf_{\combExt \ni y \to x_0} \EuHp k_j(y)
    \ge  \lim_{j \to \infty}  (k_j(x_0)-k(x_0))
    = \infty.   
\end{align*}
Thus $v \in \UU_{h}$,
and in particular
\[
    \EuHp h(z_0) \le v(z_0) \le \EuHp k(z_0) + 2\eps.
\]
Letting $\eps \to 0$ shows that $\EuHp h(z_0) \le \EuHp k(z_0)$,
and as $z_0 \in \comb$ was arbitrary we find that
$     \EuHp h \le \EuHp k $ 
in $\comb$.
It follows that 
\[ 
      \limsup_{\combExt \ni z \to x_0} {(\EuHp h(z) - U(z))} 
    \le     \limsup_{\combExt \ni z \to x_0} {(\EuHp k(z) - U(z))}
      =0,
\] 
by either Theorem~5.2 in Bj\"orn~\cite{ABjump} (if $x_0 \ne 0$) 
or Theorem~5.4 in \cite{ABjump} (if $x_0=0$).
Applying this also to $-h$ 
and using that $\ElHp h \le \EuHp h$
give \eqref{eq-lim-gen-p>2} and
complete the proof.
\end{proof}

\begin{proof}[Proof of Theorem~\ref{thm-comb-last}]
Without loss of generality we may assume that
$0 \le f \le 2$  and that $f(0)=1$.
Let 
\[
    k=\begin{cases}
        f & \text{on } \bdyP \comb, \\
        1 & \text{on } I,
      \end{cases}
      \quad \text{and} \quad
    \ktilde=k+\chi_{\{0\}}.
\]

Fix $x \in \bdyP \comb \setm \{0\}$ for the moment.
We first observe that it follows, 
from 
either Lemma~\ref{lem-key} (if $p >2$)
or Theorem~5.2 in Bj\"orn~\cite{ABjump} and Remark~\ref{rmk-mods}
(if $p \le 2$),
that  there is
a function $U_{x}: \comb \to \R$
 such that
\begin{align}
  \label{eq-f-diff}
  \lim_{\combExt \ni y \to x} (\EuHp f(y) - U_x(y)) &  =   
    \lim_{\combExt \ni y \to x} (\ElHp f(y) - U_x(y)) = 0, \\
  \label{eq-k-diff}
   \lim_{\combExt \ni y \to x} (\EuHp k(y) - U_x(y)) & = 
   \lim_{\combExt \ni y \to x} (\ElHp k(y) - U_x(y))   =0.
\end{align}
We also need that 
\begin{align}
  \label{eq-h-diff}
   \lim_{\combExt \ni y \to x} (\EuHp h(y) - U_x(y)) & =
   \lim_{\combExt \ni y \to x} (\ElHp h(y) - U_x(y)) = 0,
\end{align}
which again follows from Lemma~\ref{lem-key} if $p >2$.

To establish \eqref{eq-h-diff} for $p\le 2$ we proceed as follows:
Let $m$ be a positive integer such that $2^{1-m} < |x|$.
Let also $G=((-1,1)  \times (0,2)) \setm \bigcup_{j=0}^m \itoverline{I}_j$,
and let $\combG$ and $\bdy_{\itoverline{G}^M} \comb$
be as in the proof of the general case of Theorem~\ref{thm-main}.
We can then find a function $w: \bdy_{\itoverline{G}^M} \comb \to \R$
such that $w \ge f$ on $\bdyExt \comb$, $w$ 
is continuous at all points in $\bdy_{\itoverline{G}^M} \comb  \setm \{x\}$, 
and  $w-f$ is continuous at $x$.
Furthermore, let $E' \subset \bdy_{\itoverline{G}^M} \comb$ be the set corresponding
to $\Et$ and $\wt =w + \infty \chi_{I \cup E'}$.
Next, we need to apply Theorem~7.2 in Bj\"orn~\cite{ABjump} to the
function $\wt \ge h$, but
with respect to $\combG$.
The proof therein applies also in this case with the following remarks:
\begin{enumerate}
\item
The use of Theorem~5.2 in \cite{ABjump} is valid
also in our case, see the discussion in Remark~\ref{rmk-mods}.
\item
Instead of appealing to Theorem~2.4 in \cite{ABjump}
(which is Theorem~6.1 in Bj\"orn--Bj\"orn--Shanmugalingam~\cite{BBS2})
we need to use Theorem~11.2 % Check no
in Bj\"orn--Bj\"orn--Shanmugalingam~\cite{BBSdir}
and the fact that $\bCp(E'\cup I,\combG)=0$ (which follows from 
\eqref{eq-bCpI} and 
\eqref{eq-CpEt}),
where $\bCp$
is the new capacity introduced in \cite{BBSdir}.
\item
The proof in \cite{ABjump} is not valid for $p=2$, but
in this case the result follows more easily using linearity.
(When $p=2$ the entire Theorem~\ref{thm-comb-last} can also be 
deduced more easily using linearity.)
\end{enumerate}
We thus obtain that
\[
   \lim_{\combG \ni y \to x} (\uHpind{\combG} \wt(y) - U_x(y))  =0.
\]
By simple comparison we have $\EuHp h \le \EuHp \wt = \uHpind{\combG} \wt$,
and thus
\[
   \limsup_{\combExt \ni y \to x} (\EuHp h(y) - U_x(y)) 
   \le 0.
\]
Applying this also to $-h$ and using that $\ElHp h \le \EuHp h$ establishes
\eqref{eq-h-diff} for $p \le 2$, i.e.\ for all $p$.

Let  next $\phi:=2+\infty \chi_{I \cup \Phi(\Et)}$.
(Note that $\phi$ is a function on $\bdy \comb$.)
Then $\phi \ge h$ on $\Extbdy \comb$.
By either Theorem~9.1 % Check no
in Bj\"orn--Bj\"orn--Shanmugalingam~\cite{BBSdir}  and \eqref{eq-bCpI}
(if $p \le 2$)
or the comparison principle (Theorem~\ref{thm-comp}) 
and Lemma~\ref{lem-key} (if $p >2$), 
$\EHp \phi \equiv \Hp \phi \equiv 2$.
Let $u \in \UUt_{\ktilde}$.
Then, by simple comparison
\[
      \liminf_{\comb \ni y \to 0} u(y) \ge \ktilde(0)=2
      = \limsup_{\comb \ni y \to 0} \EHp \phi(y) 
      \ge \limsup_{\comb \ni y \to 0} \EuHp h(y). 
\]
Moreover, for $x \in \bdyP \comb \setm \{0\}$,
\[
         \liminf_{\combExt \ni y \to x} (u (y) - \EuHp h(y))
        \ge  \lim_{\combExt \ni y \to x} (\EuHp k(y) - \EuHp h(y))
        =0,
\]
by \eqref{eq-k-diff} and \eqref{eq-h-diff}.
Thus,  the comparison principle (Theorem~\ref{thm-comp}) yields that
$u \ge \EuHp h$.
Since $u \in \UUt_{\ktilde}$ was arbitrary, we obtain 
that $\EuHp \ktilde \ge \EuHp h$.

As $k$ is continuous at $0$ there is $\psi \in C(\bdy \comb)$
such that $\psi \ge k$ on $\bdyExt \comb$ and $\psi(0)=1$.
(Note that $\psi$ is a function on $\bdy \comb$.)
Let also $\psit=\psi +\chi_{\{0\}}$ so that $\psit \ge \kt$ on $\bdyExt \comb$.
Then $\Hp \psit = \Hp \psi$ by 
either Theorem~6.3 in Bj\"orn~\cite{ABjump} (if $p>2$, note
that we apply it to normal Perron solutions)
or 
Theorem~6.1 in Bj\"orn--Bj\"orn--Shanmugalingam~\cite{BBS2} (if $p \le 2$)
(which can also be found as Theorem~10.29 in \cite{BBbook};
the more general  Theorem~9.1 in % Check no
Bj\"orn--Bj\"orn--Shanmugalingam~\cite{BBSdir} can also be used).
We conclude, using also simple comparison, that
\[
     \EuHp h \le \EuHp \kt \le \Hp \psit = \Hp \psi.
\]
Hence 
\[
      \limsup_{\comb \ni y \to 0} \EuHp h(y) 
      \le \lim_{\comb \ni y \to 0} \Hp \psi(y) = 1,
\]
where the last equality holds because $0$ is regular.
Applying this to $2-h$ shows that we also have
\[
      \liminf_{\comb \ni y \to 0} \ElHp h(y) 
      \ge 1,
\]
which together with the inequality $\ElHp h \le \EuHp h$ gives that 
\[
      \lim_{\comb \ni y \to 0} \EuHp h(y) 
      =\lim_{\comb \ni y \to 0} \ElHp h(y) =1.
\]
In particular this holds when $h=f$.

For $x \in \bdyP \comb \setm \{0\}$,
we get from  \eqref{eq-f-diff} and \eqref{eq-h-diff}
that
\begin{align*}
         \lim_{\bdyExt\comb \ni y \to x} (\EuHp h(y) - \EuHp f(y)) 
  &      =         \lim_{\bdyExt\comb \ni y \to x} (\ElHp h(y) - \EuHp f(y)) \\
  &      =    \lim_{\bdyExt\comb \ni y \to x} (\ElHp h(y) - \ElHp f(y))
        =0.
\end{align*}
Thus,  the comparison principle (Theorem~\ref{thm-comp}) yields that
$\EuHp h \equiv \ElHp h \equiv \EuHp f \equiv \ElHp f$.
The inequalities in  \eqref{eq-fund} complete the proof.
\end{proof}


\begin{thebibliography}{99}

\bibitem{ABBSprime} \art{Adamowicz, T.,
        Bj\"orn, A., Bj\"orn, J. \AND Shan\-mu\-ga\-lin\-gam, N.}
        {Prime ends for domains in metric spaces}
        {Adv. Math.} {238} {2013} {459--505}

\bibitem{ABclass} \art{\idxauth{Bj\"orn}{A}}
         {A regularity classification of boundary points
           for \p-harmonic functions and quasiminimizers}
        {J. Math. Anal. Appl.} {338} {2008} {39--47}


\bibitem{ABjump} \art{Bj\"orn, A.}
        {\p-harmonic functions with boundary data having jump discontinuities
	and Baernstein's problem}
        {J. Differential Equations} {249} {2010} {1--36}

\bibitem{BB2} \art{\idxauth{Bj\"orn}{A} \AND \idxauth{Bj\"orn}{J}}
	{Approximations by regular sets and Wiener solutions in metric spaces}
	{Comment. Math. Univ. Carolin.} {48} {2007} {343--355}

\bibitem{BBbook} \book{Bj\"orn, A. \AND Bj\"orn, J.}
        {\it Nonlinear Potential Theory on Metric Spaces}
    {EMS Tracts in Mathematics {\bf 17},
        European Math. Soc., Zurich, 2011}

\bibitem{BBS} \art{\idxauth{Bj\"orn}{A}, \idxauth{Bj\"orn}{J},
	\AND \idxauth{Shanmugalingam}{N}}
        {The Dirichlet problem for \p-harmonic functions on metric spaces}
        {J. Reine Angew. Math.} {556} {2003} {173--203}

\bibitem{BBS2} \art{Bj\"orn, A., Bj\"orn, J. \AND Shanmugalingam, N.}
        {The Perron method for \p-harmonic functions}
        {J. Differential Equations} {195} {2003} {398--429}

\bibitem{BBS3} \art{\idxauth{Bj\"orn}{A}, \idxauth{Bj\"orn}{J}
	\AND \idxauth{Shanmugalingam}{N}}
        {A problem of
	Baernstein on the equality of the \p-harmonic
	measure of a set and its closure}
        {Proc. Amer. Math. Soc.} {134} {2006} {509--519}

\bibitem{BBSdir} {\sc Bj\"orn, A., Bj\"orn, J. \AND Shanmugalingam, N.},
    {The Dirichlet problem for \p-harmonic functions with respect to
the Mazurkiewicz boundary, and new capacities},
     {\it Preprint}, 2013, {\tt arXiv:1302.3887}.

\bibitem{BBStop} \artprep{Bj\"orn, A., Bj\"orn, J. 
        \AND Shan\-mu\-ga\-lin\-gam, N.}
        {The Mazurkiewicz distance and sets which
          are finitely connected at the boundary}
        {\it In preparation}

\bibitem{car}  \art{Carath\'eodory, C.}
       {\"Uber die Begrenzung einfach zusammenh\"angender Gebiete}
       {Math. Ann.} {73} {1913} {323--370}

\bibitem{ES} \artprep{Estep, D. \AND Shanmugalingam, N.}
        {The topology of the prime end boundary and the Perron method for the 
         Dirichlet problem in metric measure spaces}
        {\it In preparation}

\bibitem{GrLiMa86} \art{Granlund, S.,  Lindqvist, P. \AND  Martio, O.}
        {Note on the PWB-method in the nonlinear case}
        {Pacific J. Math.} {125} {1986} {381--395}

\bibitem{HeKiMa} \book{Heinonen, J., Kilpel\"ainen, T.\ \AND Martio, O.}
        {Nonlinear Potential Theory of Degenerate Elliptic Equations}
        {2nd ed., Dover, Mineola, NY, 2006}

\bibitem{Kilp89} \art{Kilpel\"ainen, T.}
           {Potential theory for supersolutions of degenerate elliptic equations}
           {Indiana Univ. Math. J.} {38} {1989} {253--275}

\bibitem{mazya70}
     \artnopt{\auth{Maz{\cprime}ya}{V. G}}
        {On the continuity at a boundary point of solutions of quasi-linear
        elliptic equations}
        {Vestnik Leningrad. Univ. Mat. Mekh. Astronom.}
        {25{\rm:13}} {1970} {42--55}  (Russian).
        English transl.: {\it Vestnik Leningrad Univ. Math.}
        {\bf 3} (1976), 225--242.

\bibitem{perron} \art{\idxauth{Perron}{O}}
         {Eine neue Behandlung der ersten Randwertaufgabe f\"ur
           $\Delta u =0$}
         {Math. Z.} {18} {1923} {42--54}

\bibitem{remak} \art{\idxauth{Remak}{R}}
	{\"Uber potentialkonvexe Funktionen}
	{Math. Z.} {20} {1924} {126--130}

\end{thebibliography}
\end{document}